\documentclass[11pt]{amsart}

\usepackage{graphics}
\usepackage{color}
\usepackage[a4paper, margin=3cm]{geometry}

\usepackage{amssymb,enumerate}
\usepackage{amsmath}
\usepackage{bbm}           
\usepackage{bm}
\usepackage{eso-pic,graphicx}
\usepackage{tikz}
\usepackage{cite}
\usepackage{esint}
\usepackage[colorlinks=true, pdfstartview=FitV, linkcolor=blue, citecolor=blue, urlcolor=blue]{hyperref}

\usepackage{graphicx}
\usepackage{pslatex}
\usepackage{amsmath,amsfonts}
\usepackage{rotating}
\usepackage{amssymb}
\usepackage{verbatim}
\usepackage{rotating}
\usepackage{mathrsfs}

\makeatletter
\def\Ddots{\mathinner{\mkern1mu\raise\p@
\vbox{\kern7\p@\hbox{.}}\mkern2mu
\raise4\p@\hbox{.}\mkern2mu\raise7\p@\hbox{.}\mkern1mu}}
\makeatother

\def\XXint#1#2#3{{\setbox0=\hbox{$#1{#2#3}{\int}$}
\vcenter{\hbox{$#2#3$}}\kern-.5\wd0}}

\begin{document}
\newtheorem{theorem}{Theorem}
\newtheorem{proposition}[theorem]{Proposition}
\newtheorem{conjecture}[theorem]{Conjecture}
\def\theconjecture{\unskip}
\newtheorem{corollary}[theorem]{Corollary}
\newtheorem{lemma}[theorem]{Lemma}
\newtheorem{claim}[theorem]{Claim}
\newtheorem{sublemma}[theorem]{Sublemma}
\newtheorem{observation}[theorem]{Observation}
\theoremstyle{definition}
\newtheorem{definition}{Definition}
\newtheorem{notation}[definition]{Notation}
\newtheorem{remark}[definition]{Remark}
\newtheorem{question}[definition]{Question}
\newtheorem{questions}[definition]{Questions}
\newtheorem{example}[definition]{Example}
\newtheorem{problem}[definition]{Problem}
\newtheorem{exercise}[definition]{Exercise}
 \newtheorem{thm}{Theorem}
 \newtheorem{cor}[thm]{Corollary}
 \newtheorem{lem}{Lemma}[section]
 \newtheorem{prop}[thm]{Proposition}
 \theoremstyle{definition}
 \newtheorem{dfn}[thm]{Definition}
 \theoremstyle{remark}
 \newtheorem{rem}{Remark}
 \newtheorem{ex}{Example}
 \numberwithin{equation}{section}
\def\C{\mathbb{C}}
\def\R{\mathbb{R}}
\def\Rn{{\mathbb{R}^n}}
\def\Rns{{\mathbb{R}^{n+1}}}
\def\Sn{{{S}^{n-1}}}
\def\M{\mathbb{M}}
\def\N{\mathbb{N}}
\def\Q{{\mathbb{Q}}}
\def\Z{\mathbb{Z}}
\def\F{\mathcal{F}}
\def\L{\mathcal{L}}
\def\S{\mathcal{S}}
\def\supp{\operatorname{supp}}
\def\essi{\operatornamewithlimits{ess\,inf}}
\def\esss{\operatornamewithlimits{ess\,sup}}

\numberwithin{equation}{section}
\numberwithin{thm}{section}
\numberwithin{theorem}{section}
\numberwithin{definition}{section}
\numberwithin{equation}{section}

\def\earrow{{\mathbf e}}
\def\rarrow{{\mathbf r}}
\def\uarrow{{\mathbf u}}
\def\varrow{{\mathbf V}}
\def\tpar{T_{\rm par}}
\def\apar{A_{\rm par}}

\def\reals{{\mathbb R}}
\def\torus{{\mathbb T}}
\def\scriptm{{\mathcal T}}
\def\heis{{\mathbb H}}
\def\integers{{\mathbb Z}}
\def\z{{\mathbb Z}}
\def\naturals{{\mathbb N}}
\def\complex{{\mathbb C}\/}
\def\distance{\operatorname{distance}\,}
\def\support{\operatorname{support}\,}
\def\dist{\operatorname{dist}\,}
\def\Span{\operatorname{span}\,}
\def\degree{\operatorname{degree}\,}
\def\kernel{\operatorname{kernel}\,}
\def\dim{\operatorname{dim}\,}
\def\codim{\operatorname{codim}}
\def\trace{\operatorname{trace\,}}
\def\Span{\operatorname{span}\,}
\def\dimension{\operatorname{dimension}\,}
\def\codimension{\operatorname{codimension}\,}
\def\nullspace{\scriptk}
\def\kernel{\operatorname{Ker}}
\def\ZZ{ {\mathbb Z} }
\def\p{\partial}
\def\rp{{ ^{-1} }}
\def\Re{\operatorname{Re\,} }
\def\Im{\operatorname{Im\,} }
\def\ov{\overline}
\def\eps{\varepsilon}
\def\lt{L^2}
\def\diver{\operatorname{div}}
\def\curl{\operatorname{curl}}
\def\etta{\eta}
\newcommand{\norm}[1]{ \|  #1 \|}
\def\expect{\mathbb E}
\def\bull{$\bullet$\ }

\def\blue{\color{blue}}
\def\red{\color{red}}

\def\xone{x_1}
\def\xtwo{x_2}
\def\xq{x_2+x_1^2}
\newcommand{\abr}[1]{ \langle  #1 \rangle}

\newcommand{\Norm}[1]{ \left\|  #1 \right\| }
\newcommand{\set}[1]{ \left\{ #1 \right\} }
\newcommand{\ifou}{\raisebox{-1ex}{$\check{}$}}
\def\one{\mathbf 1}
\def\whole{\mathbf V}
\newcommand{\modulo}[2]{[#1]_{#2}}
\def \essinf{\mathop{\rm essinf}}
\def\scriptf{{\mathcal F}}
\def\scriptg{{\mathcal G}}
\def\scriptm{{\mathcal M}}
\def\scriptb{{\mathcal B}}
\def\scriptc{{\mathcal C}}
\def\scriptt{{\mathcal T}}
\def\scripti{{\mathcal I}}
\def\scripte{{\mathcal E}}
\def\scriptv{{\mathcal V}}
\def\scriptw{{\mathcal W}}
\def\scriptu{{\mathcal U}}
\def\scriptS{{\mathcal S}}
\def\scripta{{\mathcal A}}
\def\scriptr{{\mathcal R}}
\def\scripto{{\mathcal O}}
\def\scripth{{\mathcal H}}
\def\scriptd{{\mathcal D}}
\def\scriptl{{\mathcal L}}
\def\scriptn{{\mathcal N}}
\def\scriptp{{\mathcal P}}
\def\scriptk{{\mathcal K}}
\def\frakv{{\mathfrak V}}
\def\C{\mathbb{C}}
\def\D{\mathcal{D}}
\def\R{\mathbb{R}}
\def\Rn{{\mathbb{R}^n}}
\def\rn{{\mathbb{R}^n}}
\def\Rm{{\mathbb{R}^{2n}}}
\def\r2n{{\mathbb{R}^{2n}}}
\def\Sn{{{S}^{n-1}}}
\def\M{\mathbb{M}}
\def\N{\mathbb{N}}
\def\Q{{\mathcal{Q}}}
\def\Z{\mathbb{Z}}
\def\F{\mathcal{F}}
\def\L{\mathcal{L}}
\def\G{\mathscr{G}}
\def\ch{\operatorname{ch}}
\def\supp{\operatorname{supp}}
\def\dist{\operatorname{dist}}
\def\essi{\operatornamewithlimits{ess\,inf}}
\def\esss{\operatornamewithlimits{ess\,sup}}
\def\dis{\displaystyle}
\def\dsum{\displaystyle\sum}
\def\dint{\displaystyle\int}
\def\dfrac{\displaystyle\frac}
\def\dsup{\displaystyle\sup}
\def\dlim{\displaystyle\lim}
\def\bom{\Omega}
\def\om{\omega}
\def\BMO{\rm BMO}
\def\CMO{\rm CMO}

\author[S. Wang]{Shifen Wang}
\address{Shifen Wang:
	School of Mathematical Sciences \\
	Beijing Normal University \\
	Laboratory of Mathematics and Complex Systems \\
	Ministry of Education \\
	Beijing 100875 \\
	People's Republic of China}
\email{wsfrong@mail.bnu.edu.cn}

\author[Q. Xue]{Qingying Xue$^{*}$}
\address{Qingying Xue:
	School of Mathematical Sciences \\
	Beijing Normal University \\
	Laboratory of Mathematics and Complex Systems \\
	Ministry of Education \\
	Beijing 100875 \\
	People's Republic of China}
\email{qyxue@bnu.edu.cn}

\keywords{Schr\"{o}dinger operator, semi-group maximal operator, commutator, compactness.\\
	\indent{\it {2010 Mathematics Subject Classification.}} Primary 42B25,
	Secondary 35J10.}

\thanks{ The second author was supported partly by NNSF of China (Nos. 11671039, 11871101) and NSFC-DFG (No. 11761131002).
	\thanks{$^{*}$ Corresponding author, e-mail address: qyxue@bnu.edu.cn}}

\date{December 3, 2020}
\title[ on weighted Compactness of Commutators ]
{\bf On weighted Compactness of Commutator of semi-group maximal function associated to Schr\"odinger operators}

\begin{abstract} 
Let $\mathcal{T}^*$ be the semi-group maximal function associated to the Schr\"odinger operator $-\Delta+V(x)$ with $V$ satisfying an appropriate reverse H\"{o}lder inequality. In this paper, we show that the commutator of $\mathcal{T}^*$ is a compact operator on $L^p(w)$ for $1<p<\infty$ if $b\in \text{CMO}_\theta(\rho)(\mathbb{R}^n)$ and $w\in A_p^{\rho,\theta}(\mathbb{R}^n)$. Here $\text{ CMO}_\theta(\rho)(\mathbb{R}^n)$ denotes the closure of $\mathcal{C}_c^\infty(\mathbb{R}^n)$ in the $\text{BMO}_\theta(\rho)(\mathbb{R}^n)$ (which is larger than the classical $\text{BMO}(\mathbb{R}^n)$ space) topology. The space where $b$ belongs and the weighs class $w$ belongs are more larger than the usual $\text{CMO}(\mathbb{R}^n)$ space and the Muckenhoupt $A_p$ weights class, respectively.

\end{abstract}\maketitle

\section{Introduction}

\medskip
Consider the Schr\"{o}dinger operator
\begin{equation*}
L=-\Delta+V(x)
\end{equation*}
in $\mathbb{R}^n,\ n\geq 3$. Where $\Delta$ is the Laplacian operator on $\mathbb{R}^n$ and the function $V$ is a nonnegative potential belonging to certain reverse H\"{o}lder class $RH_q$ with an exponent $q>n/2$, that is, there exists a constant $C>0$ such that
\begin{equation*}
\Big(\frac{1}{|B(x,r)|}\int_{B(x,r)}V^q(y)dy\Big)^{\frac{1}{q}}\leq C\Big(\frac{1}{|B(x,r)|}\int_{B(x,r)}V(y)dy\Big),
\end{equation*}
for every ball $B(x,r)\subset\mathbb{R}^n$. It is worth pointing out that if $V\in RH_q$ for some $q>1$, then there exists $\epsilon>0$, such that $V\in RH_{q+\epsilon}$ (see \cite{G}). On the other hand, the H\"older inequality gives that $RH_{q_1}\subset RH_{q_2}$ if $q_1\geq q_2>1$. Therefore, the assumption $q>n/2$ is equivalent to the case $q\geq n/2$. Thoughout this paper, we always assume that $V \not\equiv0$ and $V\in RH_q$ with $q\geq n/2$. 

The Schr\"odinger operator $L$ with nonnegative potentials is very useful in the study of certain subelliptic operators. For instance, by taking the partial Fourier transform in the $t$ variable, the operator $-\Delta_x-V(x)\partial^2_t$ is reduced to $-\Delta_x+V(x)\xi^2$ (See \cite{S}). Some basic results on $L$, including certain estimates of the fundamental solutions of $L$ and the boundedness on $L^p$ of Riesz transforms $\Delta L^{-1/2}$ were obtained by Fefferman \cite{F}, Shen \cite{S1} and Zhong \cite{Z}. 

Attentions have also been paid to the study of function spaces associated to $L$.
It was Dziuba\'nski and Zienkiewicz \cite{DZ2} who characterized the Hardy space $H^1_L$ related to the Schr\"odinger operator. Later on, for $0<p\leq 1$, the $H^p_L$ space with potentials from reverse H\"older classes were studied in \cite{DZ} and \cite{DZ1}. Subsequently, Yang et al. \cite{YYZ} characterize the localized Hardy spaces by establishing the boundedness of  Riesz transforms, maximal operators and endpoint estimates of fractional integrals associated with $L$.

For the classical Schr\"odinger operators $L$ , there are many interesting results of its associated Riesz transforms, which essentially and heavily depend on the properties of $e^{-tL}$. The properties of semi-group $e^{-tL}$ such as the positivity, Gaussian estimates and off-diagonal estimates play a fundamental role in the study of Riesz transform. The maximal function defined by the semigroup $e^{-tL}\ (t>0)$ or the Riesz transforms $\Delta L^{-1/2}$ were further generalized by Lin, et. al. \cite{LLLY} to the setting of Heisenberg groups.

In order to introduce more results, we need to give some definitions.
The semi-group maximal function associated to the Schr\"odinger operator $L$ is defined by
\begin{equation}\label{1.1}
\mathcal{T}^*(f)(x)=\sup_{t>0}|e^{-tL}f(x)|=\sup_{t>0}\Big|\int_{\mathbb{R}^n}k_t(x,y)f(y)dy\Big|,
\end{equation}
where $k_t$ is the kernel of the operator $e^{-tL}$.
As in \cite{S1}, we will use the auxiliary function $\rho$ defined for $\mathbb{R}^n$ as
\begin{equation}\label{1.2}
\rho(x)=\sup_{r>0}\Big\{r:\frac{1}{r^{n-2}}\int_{B(x,r)}V(y)dy\leq1\Big\}.
\end{equation}

\remark
Under the above assumptions on $V$, it is easy to get that $0<\rho(x)<\infty$. In particular, $\rho(x)=1$ with $V=1$. For more details concerning the function $\rho(x)$ and its applications in studying the Schr\"odinger operator $L$, we refer the reader to \cite{F,T,S1}.

For $1<p<\infty$, the $A_p^{\rho,\theta}$ weights class is defined as follows.
\begin{definition}{\bf ($A_p^{\rho,\theta}$ weights class, \cite{BHS1})}.\label{def2.1}
Let $w$ be a nonnegative, locally integrable function on
$\mathbb{R}^n$. For $1<p<\infty$, we say that a weight $w$ belongs to the class $A_p^{\rho,\theta}$ if there exists a positive constant $C$ such that for all balls $B=B(x,r)$, it holds that
\begin{equation}\label{1.3}
\Big(\frac{1}{|B|}\int_Bw(y)dy\Big)\Big(\frac{1}{|B|}\int_Bw(y)^{-1/(p-1)}dy\Big)^{p-1}\leq C\Big(1+\frac{r}{\rho(x)}\Big)^{\theta p}.
\end{equation}
\end{definition}
$w$ is said to satisfy the $ A_1^{\rho,\theta}$ condition if there exists a constant $C$ such that for all balls $B$
\begin{equation*}
M^\theta_V(w)(x)\leq Cw(x),\ a.e. \ x\in\mathbb{R}^n,
\end{equation*}
where
\begin{equation*}
M^\theta_V(f)(x)=\sup_{x\in B}\frac{1}{\Big(1+\frac{r}{\rho(x)}\Big)^{\theta }|B|}\int_B|f(y)|dy.
\end{equation*}

\remark Clearly, the classes $A_p^{\rho,\theta}$ are increasing with $\theta$, and $A_p\subset A_p^{\rho,\theta}$ for $1\leq p<\infty$. Moreover, from the Remark \ref{rem1.1} below, it is easy to see that $A_p\subsetneq A_p^{\rho,\theta}$. 

\begin{definition}{\bf ($\BMO_\theta(\rho)(\rn)$ space, \cite{BHS})}. For $\theta>0$, 
we defined the class $\BMO_\theta(\rho)(\rn)$ of locally integrable functions $f$ such that 
\begin{equation}\label{1.4}
\frac{1}{|B(x,r)|}\int_{B(x,r)}|f(y)-f_B|dy\leq C\Big(1+\frac{r}{\rho(x)}\Big)^{\theta },
\end{equation}
for all $x\in\mathbb{R}^n$ and $r>0$, where $f_B=\frac{1}{|B|}\int_Bb$. A norm for $f\in \BMO_\theta(\rho)(\rn)$, denoted by $\|f\|_{\BMO_\theta(\rho)}$ , is given by the infimum of 
the constants satisfying \eqref{1.4}, after identifying functions that differ upon a constant. 
Clearly $\BMO(\rn)\subset \BMO_\theta(\rho)(\rn)\subset \BMO_{\theta'}(\rho)(\rn)$ for $0<\theta<\theta'$.
\end{definition}
The commutator of $\mathcal{T^*}$ with $b\in\BMO_\theta(\rho)(\rn)$ is defined by
\begin{equation}\label{1.5}
\mathcal{T}_b^*(f)(x)=\sup_{t>0}\Big|\int_{\mathbb{R}^n}k_t(x,y)(b(x)-b(y))f(y)dy\Big|.
\end{equation}

In 2011, Bongioanni, Harboure and Salinas \cite{BHS} considered the $L^p(\mathbb{R}^n)(1 <p<\infty)$ boundedness of the commutators of Riesz transforms related to $L$ with $\BMO_\theta(\rho)(\rn)$ functions. In another paper, they \cite{BHS1} established the weighted boundedness for the semi-group maximal function, Riesz transforms, fractional integrals and Littlewood-Paley functions related to $L$ with weights belong to $A_p^{\rho,\theta}$ (see definition \ref{def2.1}) class which includes the Muckenhoupt weight class $A_p$. 

\medskip
\quad\hspace{-20pt}{\bf Theorem A} (\cite{BHS1}). {\it  For $1<p<\infty$, the operators $\mathcal{T}^*$ is bounded on $L^p(w)$ when $w\in A_p^{\rho,\theta}$.}

\medskip

Recently, Tang \cite{T}  considered the weighted norm inequalities for $\mathcal{T}_b^*$.

\medskip
\quad\hspace{-20pt}{\bf Theorem B} (\cite{T}). {\it  Let $1<p<\infty$, $w\in A_p^{\rho,\theta}$ and $b\in\BMO_\theta(\rho)(\rn)$, then there exists a constant $C$ such that
	
\begin{equation*}
\|\mathcal{T}^*_b(f)\|_{L^p(w)}\leq C\|b\|_{\BMO_\theta(\rho)}\|f\|_{L^p(w)}.
\end{equation*}}

\smallskip

This paper is devoted to studying the weighted compactness for commutators of 
semi-group maximal function related to Schr\"odinger operators. Before stating our 
main results, we recall some background for the compactness of the commutators 
of some classical operators. Given a locally integrable function $b$, 
the commutator $[b,T]$ is defined by $$[b,T](f)(x)=bTf(x)-T(bf)(x).$$ In 1978, Uchiyama \cite{Uch} 
first studied the compactness of commutators and showed that the commutator $[b,T_\Omega]$ 
is compact on $L^p(\mathbb{R}^n)$ for $1<p<\infty$ if and only if 
$b\in{\rm CMO}(\mathbb{R}^n)$, where $T_\Omega$ is a singular integral operator 
with rough kernel $\Omega\in {\rm Lip}_1({\rm S}^{n-1})$ and ${\rm CMO}(\mathbb{R}^n)$ 
is the closure of $\mathcal{C}_c^\infty(\mathbb{R}^n)$ in the ${\rm BMO}(\mathbb{R}^n)$ 
topology. 

Since then, the study on the compactness of commutators of different operators 
has attracted much more attention.  Krantz and Li applied the compactness characterization 
of the commutator $[b,T_\Omega]$ to study Hankel type operators on Bergman space in 
\cite{KL1} and \cite{KL2}.  Wang\cite{W} showed the compactness of  the commutator 
of fractional integral operator form $L^p(\rn)$ to $L^q(\rn)$.  In 2009, Chen and Ding 
\cite{CD1} proved the commutator of singular integrals with variable kernels is compact 
on $L^p(\rn)$ if and only if $b\in\CMO(\mathbb{R}^n)$ and they also establised the compactness of
Littlewood-Paley square functions in \cite{CD2}. Later on, Chen, Ding and Wang \cite{CDW} 
obtained the compactness of commutators for Marcinkiewicz Integral in Morrey Spaces. 
Recently, Liu, Wang and Xue \cite{LWX} showed the compactness of the commutator of oscillatory singular 
integrals with H\"{o}lder class kernels of non-convolutional type. We refer the reader to \cite{LOPTT,X,BT,BDMT,XYY,TX,TXYY} for the compactness of commutators of 
multilinear operators.

The above compactness results are all concerned with the space $\text {CMO}(\mathbb{R}^n)$. However, Theorem B shows that the $L^p$ boundedness holds for more larger space $\text{BMO}_\theta(\rho)(\mathbb{R}^n)$, rather than $\text{BMO}(\mathbb{R}^n)$ and the weights class $A_p^{\rho,\theta}$ is more larger than $A_p$ weights class. Let $\text{ CMO}_\theta(\rho)(\mathbb{R}^n)$ be the closure of $\mathcal{C}_c^\infty(\mathbb{R}^n)$ in the $\BMO_\theta(\rho)(\mathbb{R}^n)$ topology. Then, it is quite natural to ask the following question:
\begin{question}
Is the operator $\mathcal{T}_b^*$ compact from  $L^p(w)$ to $L^p(w)$ when $w\in A_p^{\rho,\theta}$ and $b$ belongs to the space $\text{ CMO}_\theta(\rho)(\mathbb{R}^n)$? 
\end{question}
The main purpose of this paper is to give a firm answer to the above question.
Our result is as follows:

\medskip
\begin{theorem}\label{thm1.1}
Let $1<p<\infty$, $w\in A_p^{\rho,\theta}$ and $b\in\CMO_\theta(\rho)(\rn)$. If $w$ satisfies the following condition
\begin{equation}\label{1.6}
\lim_{A\rightarrow+\infty}A^{-np+n}\int_{|x|>1}\frac{w(Ax)}{{|x|}^{np}}dx=0,
\end{equation}
then the operator $\mathcal{T}^*_b$ defined by \eqref{1.1} is a compact operator from $L^p(w)$ to $L^p(w)$.
\end{theorem}

\begin{remark}\label{rem1.1} We give some comments about Theorem \ref{thm1.1}:
\begin{enumerate}
\item The weights class in Theorem \ref{thm1.1} is more larger than the classical Muckenhoupt weights class $A_p$. In fact, if $w\in A_p$, the classical Muckenhoupt weights class, then the condition \eqref{1.6} holds. Let $0<\gamma<\theta$ and $w(x)=(1+|x|)^{-(n+\gamma)}$, it is easy to see that $w$ satisfies \eqref{1.6} and $w(x)\notin A_p$ ($1\leq p<\infty$), but $w\in A_1^{\rho,\theta}\subset A_p^{\rho,\theta}$ ($1<p<\infty$) provided that $V=1$ (see\cite{T}).
\item Obviously, the space $\CMO_\theta(\rho)(\rn)$ where $b$
belongs is more larger than $\CMO(\rn)$ space. 
\end{enumerate}
\end{remark}

The paper is organized as follows. In section \ref{S2} we give some definitions and preliminary lemmas, which are the main ingredients of our proofs. In section \ref{S3} we will give the proof of Theorem \ref{thm1.1} via smooth truncated techniques. The domain of integration will be divided into several cases. In actuality some cases are combinable, but various subcases also arise, which increases the difficulty we need to deal with.

Throughout the paper,  the letter $C$ or $c$, sometimes with certain parameters, 
will stand for positive constants not necessarily the same one at each occurrence, 
but are independent of the essential variables.  $A\sim B$ means that there exists 
constants $C_1>0$ and $C_2>0$ such that $C_2B\leq A\leq C_1B$.

\bigskip

\section{Preliminaries }\label{S2}
\medskip
We first recall some notation. Given
a Lebesgue measurable set $E\subset\mathbb{R}^n$, $|E|$ will denote the Lebesgue measure of $E$. If $B = B(x, r)$ is a ball in $\mathbb{R}^n$ and $\lambda$ is a real number, then $\lambda B$ shall stand for the ball with the same center as $B$ and radiu $\lambda$ times that of $B$. A weight $w$ is a non-negative measurable function on $\mathbb{R}^n$. The measure associated with $w$ is the set function given by $w(E)=\int_E wdx$. For $0<p<\infty$ we denote by $L^p(w)$ the space of all Lebesgue measurable function $f(x)$ such that
$$\|f\|_{L^p(w)}=\Big(\int_{\mathbb{R}^n}|f(x)|^pw(x)dx|\Big)^{1/p}.$$

The auxiliary function $\rho$ enjoys the following property.
\begin{lemma}\label{lem2.1}	{\rm (\rm\cite{S1})}.
There exists $k_0\geq1$ and $C>0$ such that for all $x,\ y\in\rn$,
\begin{equation*}
C^{-1}\rho(x)\Big(1+\frac{|x-y|}{\rho(x)}\Big)^{-k_0}\leq\rho(y)\leq C\rho(x)(1+\frac{|x-y|}{\rho(x)}\Big)^{\frac{k_0}{k_0+1}}.
\end{equation*}
In particular, $\rho(x)\sim\rho(y)$ if $|x-y|<C\rho(x)$.
\end{lemma}

$A_p^{\rho,\theta}$ weights class has some properties analogy to $A_p$ weights class for $1\leq p<\infty$.
\begin{lemma}\label{lem2.2}	{\rm (\rm \cite{BHS1}\cite{T1})}.
Let $1<p<\infty$ and $w\in A_p^{\rho,\infty}=\bigcup_{\theta\geq0}A_p^{\rho,\theta}$. Then
\begin{enumerate}[{\rm (i)}]
\item If $1\leq p_1<p_2<\infty$, then $A_{p_1}^{\rho,\theta}\subset A_{p_2}^{\rho,\theta}.$
\item $w\in A_p^{\rho,\theta}$ if and only if $w^{-\frac{1}{p-1}}\in A_{p'}^{\rho,\theta}$, where $1/p+1/p'=1$.
\item If  $w\in A_p^{\rho,\infty}$, $1<p<\infty$,  then there exists $\epsilon>0$ such that $w\in A_{p-\epsilon}^{\rho,\infty}$.
\end{enumerate}
\end{lemma}

It should be pointed out that {\rm(iii)} of Lemma \ref{lem2.2} was proved by Bongioanni, Harboure and Salinas in \cite{BHS1}.

\medskip
For convenience, we write $\Psi_\theta(B)=(1+r/\rho(x_0))^\theta$, if $B=B(x,r_0)$. Then $M_V^\theta$ can be rewritten as
\begin{equation*}
M^\theta_V(f)(x)=\sup_{x\in B}\frac{1}{\Psi_\theta(B)|B|}\int_B|f(y)|dy,
\end{equation*}
and the following result holds:

\begin{lemma}\label{lem2.3}	{\rm (\rm\cite{T1})}.
Let $1<p<\infty$ and suppose that $w\in A_p^{\rho,\theta}$. If $p<p_1<\infty$,
then 
\begin{equation*}
\int_{\mathbb{R}^n}|M_V^\theta f(x)|^{p_1}w(x)dx\leq C_p\int_{\mathbb{R}^n}|f(x)|^{p_1}w(x)dx.
\end{equation*}
\end{lemma}

By the Lemma \ref{lem2.3},  $M_V^\theta$ may not be bounded on $L^p(w)$ for all $w\in A_p^{\rho,\theta}$ and $1<p<\infty$. So we need the variant maximal operator $M_{V,\eta}$ defined by
\begin{equation*}
M_{V,\eta} f(x)=\sup_{x\in B}\frac{1}{(\Psi_\theta(B))^\eta|B|}\int_B|f(y)|dy,\ \ \ 0<\eta<\infty.
\end{equation*}

We have the following Lemma.
\begin{lemma}\label{lem2.4}	{\rm (\rm\cite{T1})}.
Let $1<p<\infty$, $p'=p/(p-1)$ and suppose that $w\in A_p^{\rho,\theta}$. Then there  exists a constant $C>0$ such that
\begin{equation*}
\|M_{V,p'}f\|_{L^p(w)}\leq C\|f\|_{L^p(w)}.
\end{equation*}
\end{lemma}

\medskip
We also need the following  properties of the kernel  $k_t$.
\begin{lemma}\label{lem2.5}	{\rm (\rm \cite{DZ},\cite{K})}.
For every $N$, there is a constant $C_N$ such that 
\begin{equation*}
0<k_t(x,y)\leq C_N t^{-\frac{n}{2}}e^{-\frac{|x-y|^2}{5t}}\Big(1+\frac{\sqrt{t}}{\rho(x)}+\frac{\sqrt{t}}{\rho(y)}\Big)^{-N}.
\end{equation*}
\end{lemma}

\begin{lemma}\label{lem2.6}	{\rm (\rm \cite{DZ1})}.
There exists $0 < \delta <1$ and a constant $c>0$ such that for every $N>0$ there is a constant 
$C_N>0$ so that, for all $|h|\leq\sqrt{t}$
\begin{equation*}
|k_t(x+h)-k_t(x,y)|\leq C_N\Big(\frac{|h|}{\sqrt{t}}\Big)^\delta t^{-\frac{n}{2}}e^{-\frac{c|x-y|^2}{t}}\Big(1+\frac{\sqrt{t}}{\rho(x)}+\frac{\sqrt{t}}{\rho(y)}\Big)^{-N}.
\end{equation*}
\end{lemma}

\medskip
We end this section by introducing the general weighted version of Frechet-Kolmogorov
theorems, which was proved by Xue, Yabuta and Yan in \cite{XYY}.
\begin{lemma}\label{lem2.7}	{\rm (\rm \cite{XYY})}. Let $w$ be a weight on $\mathbb{R}^n$. Assume that $w^{-1/(p_0-1)}$ is also a weight on $\mathbb{R}^n$ for some $p_0>1$. Let $0<p<\infty$ and $\mathcal{F}$ be a subset in $L^p(w)$, then $\mathcal{F}$ is sequentially compact in $L^p(w)$ if the following three conditions are satisfied:
\begin{enumerate}[{\rm (i)}]	
\item $\mathcal{F}$ is bounded, i.e.,
$\sup\limits_{f\in\mathcal{F}}\|f\|_{L^p(w)}<\infty$;
\item $\mathcal{F}$ uniformly vanishes at infinity, i.e.,
\begin{equation*}
\lim\limits_{N\rightarrow\infty}\sup\limits_{f\in\mathcal{F}}\int_{|x|>N}|f(x)|^pw(x)dx=0;
\end{equation*}
\item $\mathcal{F}$ is uniformly equicontinuous, i.e.,
\begin{equation*}
\lim\limits_{|h|\rightarrow0}\sup\limits_{f\in\mathcal{F}}\int_{\mathbb{R}^n}|f(\cdot+h)-f(\cdot)|^pw(x)dx=0.
\end{equation*}
\end{enumerate}
\end{lemma}

\bigskip

\section{Proof of Theorem \ref{thm1.1} }\label{S3}
\medskip

\begin{proof}[Proof of Theorem \ref{thm1.1}]
We shall prove Theorem \ref{thm1.1} via smooth truncated techniques. First, we introduce the following smooth truncated function. Let $\varphi\in C^{\infty}([0,\infty))$ satisfy
\begin{eqnarray}\label{3.1}
0\leq\varphi\leq1\ \ \ and\ \ \
\varphi(x)=
\begin{cases}
1,          &x\in[0,1],\\
0,          &x\in[2,\infty).
\end{cases}
\end{eqnarray}
For any $\gamma>0$, let
\begin{equation}\label{3.2}
k_{t,\gamma}(x,y)=k_t(x,y)\Big(1-\varphi(\gamma^{-1}|x-y|)\Big).
\end{equation}
Define
\begin{equation}\label{3.3}
\mathcal{T}^*_\gamma f(x)=\displaystyle \sup_{t>0}\Big|\int_{\rn}k_{t,\gamma}(x,y)f(y)dy\Big|.
\end{equation}
and
\begin{equation}\label{3.4}
\mathcal{T}^*_{b,\gamma} f(x)=\displaystyle \sup_{t>0}\Big|\int_{\rn}k_{t,\gamma}(x,y)(b(x)-b(y))f(y)dy\Big|.
\end{equation}
	
For any $b\in\mathcal{C}_c^\infty(\rn)$ and $\gamma,\theta,\ \eta>0$, by \eqref{3.2}, \eqref{3.4} and lemma \ref{lem2.5} with $N=\theta\eta$, one has
\begin{align}\label{3.5}
\begin{split}
|\mathcal{T}^*_{b} f(x)-\mathcal{T}^*_{b,\gamma} f(x)|&\leq C\gamma\displaystyle\sup_{t>0}\int_{|x-y|<2\gamma}t^{-n/2}e^{-\frac{|x-y|^2}{5t}}\Big(1+\frac{\sqrt{t}}{\rho(x)}\Big)^{-\theta\eta}|f(y)|dy\\
&\leq C\gamma\displaystyle\Big\{\sup_{\sqrt{t}<\gamma}\int_{|x-y|<\sqrt{t}}t^{-n/2}e^{-\frac{|x-y|^2}{5t}}\Big(1+\frac{\sqrt{t}}{\rho(x)}\Big)^{-\theta\eta}|f(y)|dy\\
&\quad+\displaystyle\sup_{\sqrt{t}<\gamma}\int_{\sqrt{t}\leq|x-y|<2\gamma}t^{-n/2}e^{-\frac{|x-y|^2}{5t}}\Big(1+\frac{\sqrt{t}}{\rho(x)}\Big)^{-\theta\eta}|f(y)|dy\\
&\quad+\displaystyle\sup_{\sqrt{t}\geq\gamma}\int_{|x-y|<2\gamma}t^{-n/2}e^{-\frac{|x-y|^2}{5t}}\Big(1+\frac{\sqrt{t}}{\rho(x)}\Big)^{-\theta\eta}|f(y)|dy\Big\}\\
&=:C\gamma\{J_1+J_2+J_3\}.
\end{split}
\end{align}
	
One may obtain
\begin{align}\label{3.6}
\begin{split}
J_1&\leq\displaystyle \sup_{\sqrt{t}<\gamma}t^{-n/2}\Big(1+\frac{\sqrt{t}}{\rho(x)}\Big)^{-\theta\eta}\int_{|x-y|<\sqrt{t}}|f(y)|dy\\
&\leq CM_{V,\eta}f(x).
\end{split}
\end{align}
and
\begin{align}\label{3.7}
\begin{split}
J_3&\leq 2^{\theta\eta}\displaystyle \sup_{\sqrt{t}\geq\gamma}\gamma^{-n}\Big(1+\frac{2\gamma}{\rho(x)}\Big)^{-\theta\eta}\int_{|x-y|<2\gamma}|f(y)|dy\\
&\leq CM_{V,\eta}f(x).
\end{split}
\end{align}
	
It remains o estimate $J_2$.  Using the estimate $e^{-s}\leq\frac{C}{s^{M/2}}$with $M>n+\theta\eta$ and splitting to annuli, it follows that
\begin{align}\label{3.8}
\begin{split}
J_1&\leq\displaystyle \sup_{\sqrt{t}<\gamma}\sum_{k=1}^{\infty}t^{\frac{M-n}{2}}\Big(1+\frac{\sqrt{t}}{\rho(x)}\Big)^{-\theta\eta}\int_{|x-y|\sim2^k\sqrt{t}}\frac{|f(y)|}{|x-y|^M}dy\\
&\leq\displaystyle \sup_{\sqrt{t}<\gamma}\sum_{k=1}^{\infty}\frac{2^{-k(M-n-\theta\eta)}}{(2^k\sqrt{t})^n\Big(1+\frac{2^k\sqrt{t}}{\rho(x)}\Big)^{\theta\eta}}\int_{|x-y|<2^k\sqrt{t}}|f(y)|dy\\
&\leq CM_{V,\eta}f(x).
\end{split}
\end{align}
	
Combing \eqref{3.8} with \eqref{3.5}, \eqref{3.6} and \eqref{3.7} may lead to
\begin{equation*}
|\mathcal{T}^*_{b} f(x)-\mathcal{T}^*_{b,\gamma} f(x)|\leq C\gamma M_{V,\eta}f(x).
\end{equation*}
Then Lemma \ref{lem2.4} with $p'\leq\eta<\infty$ gives that
\begin{equation*}
\|\mathcal{T}^*_{b}f-\mathcal{T}^*_{b,\gamma} f\|_{L^p(w)}\leq C\gamma\|f\|_{L^p(w)},
\end{equation*}
which implies that
\begin{equation}\label{3.9}
\lim_{\gamma\rightarrow 0}\|\mathcal{T}^*_{b}f-\mathcal{T}^*_{b,\gamma}f\|_{L^p(w)}=0.
\end{equation}

On the other hand, if $b\in\CMO_\theta(\rho)(\rn)$, then for any $\epsilon>0$, there exists $b_\epsilon\in\mathcal{C}_c^\infty(\rn)$ such that $\|b-b_\epsilon\|_{\BMO_\theta(\rho)}<\epsilon$, so that
\begin{equation*}
\|\mathcal{T}^*_bf-\mathcal{T}^*_{b_\epsilon} f\|_{L^p(w)}\leq\|\mathcal{T}^*_{b-b_\epsilon}f\|_{L^p(w)}\leq C\|b-b_\epsilon\|_{\BMO_\theta(\rho)}\|f\|_{L^p(w)}\leq C\epsilon.
\end{equation*}
Thus, to prove $\mathcal{T}^*_b$ is compact on 
$L^p(w)$ for any $b\in\CMO_\theta(\rho)$, it suffices to prove that $\mathcal{T}^*_b$ is compact on $L^p(w)$ for any $b\in\mathcal{C}_c^\infty(\mathbb{R}^n)$. By \eqref{3.9} and \cite{Y}, 
it suffices to show that $\mathcal{T}^*_{b,\gamma}$ is compact for any $b\in\mathcal{C}_c^\infty(\mathbb{R}^n)$ when $\gamma>0$ is small enough. To this end, for arbitrary bounded set $F$ in $L^p(w)$, let 
$$\mathcal{F}=\{\mathcal{T}^*_{b,\gamma}f:f\in F\}.$$
Then, we need to show that for $b\in \mathcal{C}_c^\infty(\mathbb{R}^n)$, $\mathcal{F}$ satisfies the conditions$
\rm (i)$-$\rm(iii)$ of Lemma \ref{lem2.7}.
	
From the definition of $k_{t,\gamma}$, we know that $0<k_{t,\gamma}(x,y)\leq k_t(x,y)$, then $\mathcal{T}^*_{\gamma}f(x)\leq\mathcal{T}^*(|f|)(x)$ and $\mathcal{T}^*_{b,\gamma}f(x)\leq\mathcal{T}^*(|f|)(x)$. Hence, the boundedness of $\mathcal{T}^*_{\gamma}$ and $\mathcal{T}^*_{b,\gamma}$ also holds. Thus, we have
\begin{equation*}
\sup\limits_{f\in F}\|\mathcal{T}^*_{b,\gamma}f\|_{L^p(w)}\leq C\sup\limits_{f\in F}\|f\|_{L^p(w)}\leq C,
\end{equation*}
which yields the fact that the set $\mathcal{F}$ is bounded.
	
Assume $b\in\mathcal{C}_c^\infty(\mathbb{R}^n)$ and $\supp(b)\subset B(0,R)$, where $B(0,R)$ is the ball of radius $R$ center at origin in $\mathbb{R}^n$. For any $|x| > A > 2R$, $w \in A_p^{\rho,\theta}$, $1<p<\infty$  and $f\in F$.  By Lemma \ref{lem2.5} and the estimate
$e^{-\frac{|x-y|^2}{5t}}\leq C\frac{t^\frac{n}{2}}{|x-y|^n}$, we have
\begin{equation*}
\begin{split}
|\mathcal{T}^*_{b,\gamma}f(x)|&\leq\displaystyle \sup_{t>0}\int_{|y|<R}k_t(x,y)|b(y)f(y)|dy\\
&\leq C\displaystyle\sup_{t>0}t^{-n/2}\Big(1+\frac{\sqrt{t}}{\rho(x)}\Big)^{-N}\int_{|y|<R}e^{-\frac{|x-y|^2}{5t}}|f(y)|dy\\
&\leq C|x|^{-n}\displaystyle \int_{|y|<R}|f(y)|dy\\
&\leq C|x|^{-n}\|f\|_{L^p(w)}\displaystyle \Big(\int_{|y|<R}w^{-p'/p}(y)dy\Big)^{1/p'}.
\end{split}
\end{equation*}
Therefore
\begin{equation*}
\begin{array}{ll}
\displaystyle\int_{|x|>A}|\mathcal{T}^*_{b,\gamma}f(x)|^pw(x)dx&\leq C\displaystyle\int_{|x|>A}\frac{w(x)}{|x|^{np}}dx\\
&\displaystyle=C\sum_{j=0}^{\infty}\int_{2^jA<|x|<2^{j+1}A}\frac{w(x)}{|x|^{np}}dx\\
&\displaystyle =CA^{-np+n}\int_{|x|>1}\frac{w(Ax)}{{|x|}^{np}}dx.
\end{array}
\end{equation*}
This together with the condition \eqref{1.6} yields that
	
\begin{equation*}
\lim_{A\rightarrow\infty}\int_{|x|>A}|\mathcal{T}^*_{b,\gamma}f(x)|^pw(x)dx=0,
\end{equation*}
whenever $f\in F$.
	
It remains to show that the set $\mathcal{F}$ is uniformly equicontinuous.
It suffices to verify that
\begin{equation}\label{3.10}
\lim\limits_{|h|\rightarrow0}\|\mathcal{T}^*_{b,\gamma}f(h+\cdot)-\mathcal{T}^*_{b,\gamma}f(\cdot)\|_{L^p(w)}=0,
\end{equation}
holds uniformly for $f\in F$.
	
In what follows, we fix $\gamma\in(0,\frac{1}{4})$ and $|h|<\frac{\gamma}{4}$. Then
\begin{align}\label{3.11}
\begin{split}
|\mathcal{T}^*_{b,\gamma}f(x+h)-\mathcal{T}^*_{b,\gamma}f(x)|&\leq\displaystyle \sup_{t>0}\int_{\mathbb{R}^n}|k_{t,\gamma}(x+h,y)-k_{t,\gamma}(x,y)||b(x+h)-b(y)||f(y)|dy\\
& \quad+\displaystyle\sup_{t>0}\int_{\mathbb{R}^n}k_{t,\gamma}(x,y)|b(x+h)-b(x)||f(y)|dy\\
&=:I(x)+II(x).
\end{split}
\end{align}
	
For $II(x)$, it holds that
\begin{equation*}
\begin{array}{ll}
&II(x)=\displaystyle|b(x+h)-b(x)|\sup_{t>0}\int_{\mathbb{R}^n}k_{t,\gamma}(x,y)|f(y)|dy\\
&\qquad\leq C\displaystyle |h|\mathcal{T}^*_\gamma(|f|)(x).
\end{array}
\end{equation*}
Then, by the $L^p(w)$-bounds of $\mathcal{T}^*_\gamma$, we have
\begin{equation}\label{3.12}
\|II\|_{L^p(w)}\leq C|h|\|f\|_{L^p(w)}.
\end{equation}
	
For $I(x)$, we decompose it into two parts
\begin{align}\label{3.13}
\begin{split}
I(x)\leq&\displaystyle \sup_{\sqrt{t}\geq|h|}\int_{\mathbb{R}^n}|k_{t,\gamma}(x+h,y)-k_{t,\gamma}(x,y)||b(x+h)-b(y)||f(y)|dy\\
&\quad+\displaystyle\sup_{\sqrt{t}<|h|}\int_{\mathbb{R}^n}|k_{t,\gamma}(x+h,y)-k_{t,\gamma}(x,y)||b(x+h)-b(y)||f(y)|dy\\
&=:I_1(x)+I_2(x).
\end{split}
\end{align}
	
\textbf {Contribution of $I_1$}.
For $I_1(x)$, if $|h|\leq\sqrt{t}$, then by lemma \ref{lem2.5} and lemma \ref{lem2.6}, we have
\begin{align}\label{3.14}
\begin{split}
|k_{t,\gamma}(x+h,y)-k_{t,\gamma}(x,y)|&\leq|k_t(x+h)-k_t(x,y)|+|k_t(x+h)-k_t(x,y)|\varphi(\gamma^{-1}|x+h-y|)\\
&\quad+k_t(x,y)|\varphi(\gamma^{-1}|x+h-y|)-\varphi(\gamma^{-1}|x-y|)|\\
&\leq C\Big(\frac{|h|}{\sqrt{t}}\Big)^\delta t^{-\frac{n}{2}}e^{-\frac{c|x-y|^2}{t}}\Big(1+\frac{\sqrt{t}}{\rho(x)}+\frac{\sqrt{t}}{\rho(y)}\Big)^{-N}\\
&\quad+C\frac{|h|}{\gamma} t^{-\frac{n}{2}}e^{-\frac{c|x-y|^2}{t}}\Big(1+\frac{\sqrt{t}}{\rho(x)}+\frac{\sqrt{t}}{\rho(y)}\Big)^{-N}.
\end{split}
\end{align}
Therefore, we have
\begin{align}\label{3.15}
\begin{split}
I_1(x)&\leq C\displaystyle\sup_{\sqrt{t}\geq|h|}\int_{\rn}\Big(\Big(\frac{|h|}{\sqrt{t}}\Big)^\delta+\frac{|h|}{\gamma}\Big) t^{-\frac{n}{2}}e^{-\frac{c|x-y|^2}{t}}\Big(1+\frac{\sqrt{t}}{\rho(x)}\Big)^{-N}\\
&\quad\times|b(x+h)-b(y)||f(y)|dy\\
&\leq C\displaystyle\sup_{\sqrt{t}\geq1}\Big\{\int_{|x-y|<\sqrt{t}}+\int_{|x-y|\geq\sqrt{t}}\Big\} \Big(\Big(\frac{|h|}{\sqrt{t}}\Big)^\delta+\frac{|h|}{\gamma}\Big)t^{-\frac{n}{2}}e^{-\frac{c|x-y|^2}{t}}\\
&\quad\times\Big(1+\frac{\sqrt{t}}{\rho(x)}\Big)^{-N}|b(x+h)-b(y)||f(y)|dy\\
&\quad+C\displaystyle\sup_{|h|\leq\sqrt{t}<1}\Big\{\int_{|x-y|<\sqrt{t}}+\int_{|x-y|\geq\sqrt{t}}\Big\}\Big(\Big(\frac{|h|}{\sqrt{t}}\Big)^\delta+\frac{|h|}{\gamma}\Big) t^{-\frac{n}{2}}e^{-\frac{c|x-y|^2}{t}}\\
&\quad\times\Big(1+\frac{\sqrt{t}}{\rho(x)}\Big)^{-N}|b(x+h)-b(y)||f(y)|dy\\
&=:I_{11}(x)+I_{12}(x)+I_{13}(x)+I_{14}(x).
\end{split}
\end{align}
Now, we are in the position to estimate the above four terms.
	
For $I_{11}(x)$, 
if $\sqrt{t}\geq1$, then $t^{-\delta/2}\leq1$. Taking $N=\theta\eta$ for any $\theta,\eta>0$, then we have
\begin{equation}\label{3.16}
\begin{split}
I_{11}(x)&\leq C\gamma^{-1}(|h|^\delta+|h|)\displaystyle\sup_{\sqrt{t}\geq1}\int_{|x-y|<\sqrt{t}}t^{-\frac{n}{2}}\Big(1+\frac{\sqrt{t}}{\rho(x)}\Big)^{-\theta\eta}|f(y)|dy\\
&\leq C\gamma^{-1}(|h|^\delta+|h|)\displaystyle\sup_{\sqrt{t}\geq1}\frac{1}{(\sqrt{t})^n(1+\frac{\sqrt{t}}{\rho(x)})^{\theta\eta}}\int_{|x-y|<\sqrt{t}}t^{-\frac{n}{2}}|f(y)|dy\\
&\leq C\gamma^{-1}(|h|^\delta+|h|) M_{V,\eta}f(x).
\end{split}
\end{equation}
	
In order to estimate $I_{12}(x)$,  we need the following ineqality: for any $M>0$, there exists a constant $C>0$, such that
\begin{equation}\label{3.17}
e^{-\frac{c|x-y|^2}{t}}\leq C\frac{t^{\frac{M}{2}}}{|x-y|^{M}}.
\end{equation}
Using \eqref{3.17} with $M>n+\theta\eta$, splitting into annuli, we obtain
\begin{equation}\label{3.18}
\begin{split}
I_{12}(x)&\leq
C\gamma^{-1}(|h|^\delta+|h|)\displaystyle\sup_{\sqrt{t}\geq1}t^{\frac{M-n}{2}}\Big(1+\frac{\sqrt{t}}{\rho(x)}\Big)^{-\theta\eta}\int_{|x-y|\geq\sqrt{t}}\frac{|f(y)|}{|x-y|^M}dy\\
&\leq
C\gamma^{-1}(|h|^\delta+|h|)\displaystyle\sup_{\sqrt{t}\geq1}\sum_{k=1}^{\infty}\frac{2^{-k(M-n)}}{(2^k\sqrt{t})^n(1+\frac{\sqrt{t}}{\rho(x)})^{\theta\eta}}\int_{|x-y|<2^k\sqrt{t}}|f(y)|dy\\
&\leq
C\gamma^{-1}(|h|^\delta+|h|)\displaystyle\sup_{\sqrt{t}\geq1}\sum_{k=1}^{\infty}\frac{2^{-k(M-n-\theta\eta)}}{(2^k\sqrt{t})^n(1+\frac{2^k\sqrt{t}}{\rho(x)})^{\theta\eta}}\int_{|x-y|<2^k\sqrt{t}}|f(y)|dy\\
&\leq C\gamma^{-1}(|h|^\delta+|h|) M_{V,\eta}f(x).
\end{split}
\end{equation}
	
If $\sqrt{t}<1$, then $t^{-\delta/2}<t^{-1/2}$. For any $\theta,\eta>0$, taking $N=\theta\eta$.
For $I_{13}(x)$, if $|h|\leq\sqrt{t}$, $|x-y|<\sqrt{t}$ and $b\in\mathcal{C}_c^\infty(\rn)$, then 
$$|b(x+h)-b(y)|\leq C|x+h-y|\leq C(|x-y|+|h|)\leq C\sqrt{t}.$$
Then, it follows that
\begin{align}\label{3.19}
\begin{split}
I_{13}(x)&\leq C|h|^\delta\displaystyle\sup_{|h|\leq\sqrt{t}<1}t^{-\frac{n+1}{2}}\Big(1+\frac{\sqrt{t}}{\rho(x)}\Big)^{-\theta\eta}\int_{|x-y|<\sqrt{t}}|b(x+h)-b(y)||f(y)|dy\\
&\quad+C\gamma^{-1}|h|\displaystyle\sup_{|h|\leq\sqrt{t}<1}t^{-\frac{n}{2}}\Big(1+\frac{\sqrt{t}}{\rho(x)}\Big)^{-\theta\eta}\int_{|x-y|<\sqrt{t}}|f(y)|dy\\
&\leq C\gamma^{-1}(|h|^\delta+|h|)\displaystyle\sup_{|h|\leq\sqrt{t}<1}t^{-\frac{n}{2}}\Big(1+\frac{\sqrt{t}}{\rho(x)}\Big)^{-\theta\eta}\int_{|x-y|<\sqrt{t}}|f(y)|dy\\
&\leq C\gamma^{-1}(|h|^\delta+|h|) M_{V,\eta}f(x).
\end{split}
\end{align}
	
For $I_{14}(x)$, if $|x-y|<2^k\sqrt{t}$, $k=1,2,\cdots$, and $|h|\leq\sqrt{t}$, $b\in\mathcal{C}_c^\infty(\rn)$, then
$$|b(x+h)-b(y)|\leq C|x+h-y|\leq C(|x-y|+|h|)\leq C2^k\sqrt{t},$$
which combining with \eqref{3.17} for $M>n+1+\theta\eta$ yields that
\begin{align}\label{3.20}
\begin{split}
I_{14}(x)&\leq C|h|^\delta \displaystyle\sup_{|h|\leq\sqrt{t}<1}t^{\frac{M-n-1}{2}}\Big(1+\frac{\sqrt{t}}{\rho(x)}\Big)^{-\theta\eta}\int_{|x-y|\geq\sqrt{t}}\frac{|f(x)||b(x+h)-b(y)|}{|x-y|^M}dy\\
&\quad+ C\gamma^{-1}|h| \displaystyle\sup_{|h|\leq\sqrt{t}<1}t^{\frac{M-n}{2}}\Big(1+\frac{\sqrt{t}}{\rho(x)}\Big)^{-\theta\eta}\int_{|x-y|\geq\sqrt{t}}\frac{|f(x)|}{|x-y|^M}dy\\
&\leq C|h|^\delta\displaystyle\sup_{|h|\leq\sqrt{t}<1}t^{\frac{M-n-1}{2}}\Big(1+\frac{\sqrt{t}}{\rho(x)}\Big)^{-\theta\eta}\sum_{k=1}^{\infty}\frac{2^k\sqrt{t}}{(2^k\sqrt{t})^M}\int_{|x-y|\sim2^{k}\sqrt{t}}|f(y)|dy\\
&\quad+C\gamma^{-1}|h|\displaystyle\sup_{|h|\leq\sqrt{t}<1}t^{\frac{M-n}{2}}\Big(1+\frac{\sqrt{t}}{\rho(x)}\Big)^{-\theta\eta}\sum_{k=1}^{\infty}\frac{1}{(2^k\sqrt{t})^M}\int_{|x-y|\sim2^{k}\sqrt{t}}|f(y)|dy\\
&\leq C\gamma^{-1}(|h|^\delta+|h|)\displaystyle\sup_{|h|\leq\sqrt{t}<1}\sum_{k=1}^{\infty}\frac{2^{-k(M-n-1-\theta\eta)}}{(2^k\sqrt{t})^n(1+\frac{2^k\sqrt{t}}{\rho(x)})^{\theta\eta}}\int_{|x-y|<2^{k}\sqrt{t}}|f(y)|dy\\
&\leq C\gamma^{-1}(|h|^\delta+|h|) M_{V,\eta}f(x).
\end{split}
\end{align}
	
Sum up \eqref{3.15}, \eqref{3.16}, \eqref{3.18}, \eqref{3.19} and  \eqref{3.20} in all, we get
\begin{equation}\label{3.21}
I_1(x)\leq C\gamma^{-1}(|h|^\delta+|h|) M_{V,\eta}f(x).
\end{equation}
\medskip
	
\textbf {Contribution of $I_2$}.
Next we will estimate $I_2(x)$. When $|x-y|<\frac{\gamma}{2}$ and $|h|<\frac{\gamma}{4}$, then $|x+h-y|<\frac{3\gamma}{4}$.
Hence
$\varphi(\gamma^{-1}|x+h-y|)=1=\varphi(\gamma^{-1}|x-y|).$
This implies
$k_{t,\gamma}(x+h,y)=0=k_{t,\gamma}(x,y).$
For $I_{2}(x)$, we decompose it as follows:
\begin{align}\label{3.22}
\begin{split}
I_{2}(x)&\leq\displaystyle\sup_{\sqrt{t}<|h|<\rho(x)}\int_{|x-y|\geq\frac{\gamma}{2}}|k_{t,\gamma}(x+h,y)-k_{t,\gamma}(x,y)||b(x+h)-b(y)||f(y)|dy\\
&\quad+\displaystyle\sup_{\sqrt{t}<|h|\atop\rho(x)\leq|h|}\int_{|x-y|\geq\frac{\gamma}{2}}|k_{t,\gamma}(x+h,y)-k_{t,\gamma}(x,y)||b(x+h)-b(y)||f(y)|dy\\
&:=I_{21}(x)+I_{22}(x).
\end{split}
\end{align}

For $I_{21}(x)$, since $|x-y|>\frac{\gamma}{2}>2|h|$ and $\sqrt{t}<|h|<\rho(x)$, then $|x+h-y|\sim|x-y|$, $|h|/\rho(x)<1$ and $(|h|/\sqrt{t})^M>1$ for any $M>0$. Choosing $M>n+1+\theta\eta$ and using Lemma \ref{lem2.5} and \eqref{3.17}, we get 
\begin{align}\label{3.23}
\begin{split}
I_{21}(x)&=\displaystyle\sup_{\sqrt{t}<|h|<\rho(x)}\int_{|x-y|\geq\frac{\gamma}{2}}|k_{t,\gamma}(x+h,y)-k_{t,\gamma}(x,y)||b(x+h)-b(y)||f(y)|dy\\
&\leq C\displaystyle \sup_{\sqrt{t}<|h|<\rho(x)}\int_{|x-y|\geq2|h|}t^{-\frac{n}{2}}\Big(\frac{|h|}{\sqrt{t}}\Big)^{M-n}e^{-\frac{c|x-y|^2}{t}}|x-y||f(y)|dy\\
&\leq C|h|^{M-n}\displaystyle \sup_{\sqrt{t}<|h|<\rho(x)}\sum_{k=2}^{\infty}\int_{|x-y|\geq2|h|}\frac{|f(y)|}{|x-y|^{M-1}}dy\\
&\leq C|h|^{M-n}\displaystyle \sup_{\sqrt{t}<|h|<\rho(x)}\sum_{k=2}^{\infty}\frac{1}{(2^k|h|)^{M-1}}\int_{|x-y|\sim2^k|h|}|f(y)|dy\\
&\leq C|h|\displaystyle \sup_{\sqrt{t}<|h|<\rho(x)}\sum_{k=2}^{\infty}\frac{2^{-k(M-n-1-\theta\eta)}}{(2^k|h|)^n2^{k\theta\eta}}\int_{|x-y|<2^k|h|}|f(y)|dy\\
&\leq C|h|M_{V,\eta}f(x).
\end{split}
\end{align}
	
Finally, it remains to consider $I_{22}(x)$. Since $b\in\mathcal{C}_c^\infty(\rn)$, $|x-y|\geq2|h|$, $\sqrt{t}<|h|$, then $|h|/\sqrt{t}>1$ and $|b(x+h)-b(x)|\leq C|x-y|$. In addition, if $|x-y|<2^l\rho(x),\ l=1,2,\cdots$, then by lemma \ref{lem2.1}, we have
$$
\rho(y)\leq C2^{\frac{k_0}{k_0+1}l}\rho(x).
$$
Then, it follows that
$$
\Big(1+\frac{\sqrt{t}}{\rho(x)}+\frac{\sqrt{t}}{\rho(y)}\Big)^{-N}+\Big(1+\frac{\sqrt{t}}{\rho(x+h)}+\frac{\sqrt{t}}{\rho(y)}\Big)^{-N}\leq 2\Big(1+\frac{\sqrt{t}}{\rho(y)}\Big)^{-N}\leq C_N\Big(1+\frac{2^{-\frac{k_0}{k_0+1}l}\sqrt{t}}{\rho(x)}\Big)^{-N}.
$$
Choosing $M, N$ such that $N>M>n+1+(k_0+1)\theta\eta$, and applying lemma \ref{lem2.5} and \eqref{3.17} again, we obtain
\begin{align}\label{3.24}
\begin{split}
I_{22}(x)&\leq C|h|\displaystyle \sup_{\sqrt{t}<|h|\atop\rho(x)\leq|h|}\int_{|x-y|\geq2\rho(x)}t^{-\frac{n+1}{2}}e^{-\frac{c|x-y|^2}{t}}\Big(1+\frac{\sqrt{t}}{\rho(y)}\Big)^{-N}|x-y||f(y)|dy\\
&\leq C|h|\displaystyle \sup_{\sqrt{t}<|h|\atop\rho(x)\leq|h|}t^{\frac{M-n-1}{2}}\sum_{l=2}^{\infty}\Big(1+\frac{2^{-\frac{k_0}{k_0+1}l}\sqrt{t}}{\rho(x)}\Big)^{-N}\int_{|x-y|\sim2^l\rho(x)}\frac{|f(y)|}{|x-y|^{M-1}}dy\\
&\leq C|h|\displaystyle \sup_{\sqrt{t}<|h|\atop\rho(x)\leq|h|}\sum_{l=2}^{\infty}\Big(\frac{\sqrt{t}}{\rho(x)}\Big)^{M-n-1}\Big(1+\frac{2^{-\frac{k_0}{k_0+1}l}\sqrt{t}}{\rho(x)}\Big)^{-N}\frac{2^{-l(M-n-1-\theta\eta)}}{(2^l\rho(x))^n2^{l\theta\eta}}
\end{split}
\end{align}
\begin{align*}
\begin{split}
&\quad\times\int_{|x-y|<2^l\rho(x)}|f(y)|dy\\
&\leq C|h|\displaystyle \sup_{\sqrt{t}<|h|\atop\rho(x)\leq|h|}\sum_{l=2}^{\infty}\frac{2^{-l(\frac{M-n-1}{k_0+1}-\theta\eta)}}{(2^l\rho(x))^n2^{l\theta\eta}}\int_{|x-y|<2^l\rho(x)}|f(y)|dy\\
&\leq C|h|M_{V,\eta}f(x).
\end{split}
\end{align*}
	
Inequality \eqref{3.24} together with \eqref{3.22} and \eqref{3.23} gives that
\begin{equation}\label{3.25}
I_2(x)\leq C|h|M_{V,\eta}f(x).
\end{equation}
Therefore, by \eqref{3.13} and \eqref{3.21} we have
\begin{equation*}
I(x)\leq C(|h|+|h|^\delta)M_{V,\eta}f(x).
\end{equation*}
By Lemma \ref{lem2.4} for any $p'\leq\eta<\infty$, it holds that 
\begin{equation}\label{3.26}
\|I\|_{L^p(w)}\leq C(|h|+|h|^\delta)\|M_{V,\eta}f\|_{L^p(w)}\leq C(|h|+|h|^\delta)\|f\|_{L^p(w)}.
\end{equation}
From \eqref{3.11}, \eqref{3.12} and \eqref{3.26}, we get
$$
\|\mathcal{T}^*_{b,\gamma}f(h+\cdot)-\mathcal{T}^*_{b,\gamma}f(\cdot)\|_{L^p(w)}\leq C(|h|+|h|^\delta)\|f\|_{L^p(w)},
$$
which yields \eqref{3.10} and finishes the proof of Theorem \ref{thm1.1}.
\end{proof}

\end{document}